\documentclass[11pt,leqno]{amsart}
\usepackage{booktabs}
\usepackage{amssymb}
\usepackage{microtype}
\usepackage[hidelinks]{hyperref}

\theoremstyle{plain}
\newtheorem{theorem}{Theorem}
\newtheorem{lemma}[theorem]{Lemma}
\newtheorem{proposition}[theorem]{Proposition}
\newtheorem{corollary}[theorem]{Corollary}
\theoremstyle{definition}
\newtheorem{definition}[theorem]{Definition}
\renewcommand{\emptyset}{\varnothing}

\DeclareMathOperator{\height}{height}
\DeclareMathOperator{\Tow}{Tow}
\DeclareMathOperator{\Max}{Max}
\DeclareMathOperator{\F}{F}
\DeclareMathOperator{\did}{did}
\DeclareMathOperator{\mult}{mult}
\DeclareMathOperator{\comp}{Comp}
\DeclareMathOperator{\Conn}{Conn}
\newcommand{\one}{\mathbf{1}}
\newcommand{\N}{\mathbb{N}}
\newcommand{\Nzero}{\mathbb{N}_0}
\DeclareMathOperator{\DirInd}{DirInd}
\DeclareMathOperator{\MaxInt}{MaxInt}
\DeclareMathOperator{\Der}{Der}
\DeclareMathOperator{\Red}{Red}
\DeclareMathOperator{\m}{m}
\DeclareMathOperator{\p}{p}
\DeclareMathOperator{\Z}{Z}

\hypersetup{
  pdftitle={The A-to-A problem for finite ordered sets},
  pdfauthor={Jonathan David Farley and George Gr\"atzer}
}

\begin{document}

\title[The $A^A$ problem for finite ordered sets]
  {The $A^A$ problem for finite ordered sets}

\author[J.\,D. Farley]{Jonathan David Farley}
\address{Department of Mathematics, Morgan State University,
  1700 E. Cold Spring Lane, Baltimore, MD 21251, USA}
\email{lattice.theory@gmail.com}

\author[G. Gr\"atzer]{George Gr\"atzer}
\address{Department of Mathematics, University of Manitoba,
  Winnipeg, MB R3T 2N2, Canada}
\email{gratzer@mac.com}

\date{November 14, 2026}

\begin{abstract}
For ordered sets $X$ and $Y$, let $X^Y$ denote the ordered set of
isotone maps from $Y$ to $X$, ordered pointwise.
A~question originating in Duffus's 1978 thesis and explicitly raised
in 1984 asks whether $A^A\cong B^B$ implies $A\cong B$ for finite ordered sets $A$ and $B$.

We prove that it does. We construct a~compatible selection of intervals
of maximum height and an operation that records the exponents in
indecomposable normal forms. A~strict decrease of height reduces the
reconstruction to two product cancellations.

We also prove reconstruction for a~specified recursive family of
iterated exponentials, including every fixed left-associated and
right-associated power. As further applications, we treat repeated
self-exponentiation and every exponential term with at most four
occurrences of the variable.
\end{abstract}

\maketitle

\section{Introduction}\label{S:Introduction}
\subsection{The question and its history}\label{S:History}

As in \cite{GRAbook}, we take an ordered set to be nonempty.
For ordered sets $X$ and $Y$, let $X^Y$ be the set of isotone
(order-preserving) maps from $Y$ to $X$, with the pointwise ordering:
\[
   f\leq g \quad\text{if}\quad f(y)\leq g(y)
   \text{ for every }y\in Y.
\]
The $A^A$ problem asks whether $A^A$ determines a~finite ordered set
$A$ up to isomorphism.
This question originates in Dwight Duffus's 1978 thesis \cite{DufGH}
and is explicitly raised in \cite[p.~90]{DufHD}.

We answer this question in the affirmative.

\begin{theorem}\label{T:main}
Let $A$ and $B$ be finite ordered sets.
If $A^A\cong B^B$, then $A\cong B$.
\end{theorem}

D.~Duffus and R.~Wille \cite{DufWilGI} proved this theorem in 1979
for connected ordered sets. J.\,D. Farley \cite{FarBC}
proved that it is enough for one of them to be connected.
The height-one case and further special cases are considered in
\cite{FarBF}. For further background, see \cite{FarBBa, FarBC}.

\subsection{The idea of the proof}

The constant maps form a~copy of $A$ inside $A^A$. An order
isomorphism $A^A\cong B^B$, however, need not carry constant maps
to constant maps. For example, let $A$ be a~two-element antichain.
Then $A^A$ is a~four-element antichain, so every permutation of its
elements is an order automorphism. In particular, an automorphism
can interchange a~constant map with the identity map.

We recover $A$ by constructions that depend only on the order of
$A^A$. We are given only this order, not the composition of its
elements.

The proof relies on two constructions: $\Max$ and $\Tow$.
We define $\Max(A^A)$ by selecting an interval $[f,g]$ in $A^A$
whose height equals the height of $A^A$. This selection rule respects
products and exponentiation up to isomorphism.

We then apply $\Tow$ to $\Max(A^A)$. To describe this construction,
let $\Conn(P,Y)$ denote the union of the connected components of
$P^Y$ that contain an isotone map constant on each connected
component of $Y$.

For each nontrivial directly indecomposable factor $D$ of
$\Max(A^A)$, write its normal form
$D\cong\Conn(E_D,F_D)$. The construction $\Tow$ takes
the disjoint sum of the ordered sets $F_D$, with one
summand for each factor. The two constructions give
\[
   \Tow(\Max(A^A))
      \cong A\times\Tow(\Max(A)).
\]
For an ordered set $A$ of positive height, this
construction, applied to the known ordered set $A^A$,
therefore exposes $A$ as a product factor.

We still have to identify the other factor, $\Tow(\Max(A))$,
which is nonempty. Two further constructions, together with the
strict decrease of height under $\Tow$, identify a~product
decomposition from which cancellation recovers $\Max(A)$.
This determines $\Tow(\Max(A))$, and a~second cancellation in the
product decomposition above recovers $A$.
Proposition~\ref{P:reconstruction} gives the details.
Antichains are treated separately by counting their self-maps.

Throughout, \emph{reconstruction} means determination of an
isomorphism type. We discuss the separate question of computational
complexity in Section~\ref{S:questions}.

\subsection{Iterated exponentials}\label{S:iterated}

Theorem~\ref{T:iteration-reconstruction} applies the same method to
a~specified recursive family. Starting with $E_0(A)=A$, we put
\[
   E_i(A)=A^{\,A^{\times r_i}\times E_{i-1}(A)^{\times s_i}},
   \qquad r_i\geq0,\quad s_i\geq1.
\]
Here $X^{\times k}$ is the direct product of $k$ copies of $X$.
For every fixed depth and fixed sequence of parameters, the theorem
recovers $A$ from the last term. In particular, it recovers $A$ from
each fixed right-associated power
$A^A$, $A^{(A^A)}$, $A^{(A^{(A^A)})}$, and so on.
Left-associated powers are also included. We give short additional
arguments for repeated self-exponentiation and for every exponential
term with at most four occurrences of the variable.

These are the extensions proved in this paper. None of the
reconstruction results of the companion paper \cite{GRAA} is used
in our proofs. We discuss its scope and the general question in
Section~\ref{S:questions}.

\subsection{The organization of the paper}

Section~\ref{S:basic} collects the preliminary results, including
the normal form provided by McKenzie's structure theory.
Section~\ref{S:Lambda} defines the construction $\Tow$
and illustrates it on chains.
In~Sections~\ref{S:ordering} and~\ref{S:intervals}, we construct
the ordering used to select $\Max$ and prove its compatibility
with products and exponentiation.
Section~\ref{S:reconstruction} proves Theorem~\ref{T:main} and gives
worked reconstructions. Section~\ref{S:iteration} discusses the recursive
family introduced in Section~\ref{S:iterated} and its applications.
Section~\ref{S:questions} collects two further questions.

\section{Some basic concepts and results}\label{S:basic}

Unless we say otherwise, all ordered sets are finite.
The auxiliary chains used to order isomorphism types in
Section~\ref{S:ordering} may be infinite.
An ordered set is \emph{nontrivial} if it has more than one element.
An ordered set is \emph{connected} if its comparability graph
is connected. A~\emph{comparability path} is a~finite sequence of
elements in which successive elements are comparable. Repeated
vertices that arise under a~projection or evaluation may be deleted.
Two elements belong to the same connected component exactly when
a~comparability path joins them.
We write $\one$ for the one-element ordered set and $C_n$ for the
$n$-element chain, where $n\geq1$; thus $C_1=\one$.
The sum $P+Q$ is the disjoint union, with no comparabilities between
the two summands. The empty sum is $\emptyset$.
It contributes no summand, and its product with an ordered set is
again $\emptyset$.

For $k\geq0$, write $X^{\times k}$ for the direct product of $k$ copies
of $X$, and $kX$ for the sum of $k$ copies of $X$.
We set $X^{\times0}=\one$ and $0X=\emptyset$.
In every indexed sum or product, each index contributes one summand
or factor. Isomorphic entries are counted with their multiplicities.

\subsection{Products and height}

The \emph{height}, $\height(P)$, of an ordered set $P$ is the largest
value of $|C|-1$, where $C$ is a~chain in $P$.
For ordered sets $P$ and~$Q$,
\begin{equation}\label{E:height-product}
\begin{aligned}
   \height(P\times Q)&=\height(P)+\height(Q),\\
   \height(P+Q)&=\max\{\height(P),\height(Q)\}.
\end{aligned}
\end{equation}
To prove the product formula, observe that every strict step in a~chain
increases at least one coordinate. There can be at most
$\height(P)+\height(Q)$ such steps. Interleaving chains of maximum
length in the two factors attains this bound.
A~chain in a~sum lies in one summand, which proves the second formula.

A~product is connected if and only if both factors are connected.
Projection gives a~comparability path in each factor after repeated
vertices are deleted. Conversely, we can follow paths in the two
factors one coordinate at a~time.

A~nontrivial ordered set $D$ is \emph{directly indecomposable} if
$D\cong P\times Q$ implies that $P$ or $Q$ has one element.
By Hashimoto's Refinement Theorem \cite{HasEA}, a~connected ordered
set $P$ has a~decomposition
\begin{equation}\label{E:direct}
   P\cong D_1\times\cdots\times D_k
\end{equation}
into connected, nontrivial, directly indecomposable factors.
The factors are unique up to isomorphism and permutation.
For $P=\one$, we use the empty product, so $k=0$.

We also need cancellation for products, without a~connectedness
assumption. This result appears in
L.~Lov\'asz \cite[(4.3)]{LovFG} and
R. McKenzie \cite[Lemma~2.1]{McKJJ}.

\begin{lemma}[Cancellation]\label{L:cancellation}
Let $P$, $Q$, and $R$ be finite ordered sets.
If $P\times R\cong Q\times R$, then $P\cong Q$.
\end{lemma}

\begin{proof}
Let $\DirInd$ denote the set of isomorphism types of connected,
nontrivial, directly indecomposable ordered sets. For every
$D\in\DirInd$, introduce an indeterminate $x_D$.
Represent a~connected ordered set $K$ by the monomial
\[
   \m_K=\prod_{D\in\DirInd}x_D^{\mult_D(K)},
\]
where $\mult_D(K)$ is the multiplicity of $D$ in the direct decomposition
of $K$. In particular, $\m_{\one}=1$.
For a~finite ordered set $X$, let $\comp(X)$ be its set of connected
components, and define
\[
   \p_X=\sum_{K\in\comp(X)}\m_K.
\]
Isomorphic components are counted separately in this sum.
By the uniqueness of connected components and Hashimoto's Refinement
Theorem, $\p_X=\p_{X'}$ if and only if $X\cong X'$.
The components of a~product are the products of components, so
$\p_{X\times Y}=\p_X\p_Y$.
The hypothesis therefore gives $\p_P\p_R=\p_Q\p_R$.
Only finitely many indeterminates occur, and $\p_R\neq0$.
We can cancel $\p_R$ in the polynomial ring over $\mathbb Z$ to obtain
$\p_P=\p_Q$, and therefore $P\cong Q$.
\end{proof}

\subsection{The operation \texorpdfstring{$\Conn$}{Conn}}

For ordered sets $P$ and $Y$, let $\Conn(P,Y)$ be the union of the
connected components of $P^Y$ that contain a~map constant on each
connected component of $Y$. Such a~map may take different values
on different components of $Y$. Moreover, $\Conn(P,Y)$ contains the
entire components containing these maps, not just the maps themselves.
This is the construction denoted by $\mathcal C(P^Y)$ in
\cite{FarBC}.

If $Y$ has $r$ connected components, the maps constant on each
component form a~suborder isomorphic to $P^{\times r}$.
If $P$ is connected, this suborder is connected. Thus $\Conn(P,Y)$
is the single component of $P^Y$ containing it.

\smallskip
\noindent\emph{Example (More than the constant maps).}
Let $P=Y=C_2=\{0<1\}$. Writing a~map as the pair of its values
at $0$ and $1$, the three isotone maps are
\[
   (0,0)<(0,1)<(1,1).
\]
Only the first and last are constant, but all three belong to
$\Conn(C_2,C_2)\cong C_3$.

\smallskip
\noindent\emph{Example (A~disconnected exponent).}
Let $P$ and $Y$ both be two-element antichains. Every map $Y\to P$
is isotone and is constant on each component of $Y$, since these
components are singletons. Therefore $\Conn(P,Y)=P^Y$, a~four-element
antichain. Selecting only components containing globally constant
maps would retain just two of its elements.

\smallskip
If $P$ is bounded, then every map in $P^Y$ lies between the constant
maps with values $0_P$ and~$1_P$; hence
\begin{equation}\label{E:bounded}
   \Conn(P,Y)=P^Y.
\end{equation}

\smallskip
\noindent\emph{Example (The selected components need not exhaust $P^Y$).}
Let $P$ have two minimal elements $a,b$, two maximal elements $c,d$,
and exactly the strict comparisons
\[
   a<c,\qquad a<d,\qquad b<c,\qquad b<d.
\]
The identity map is isolated in $P^P$. Indeed, if
$f\leq\operatorname{id}_P$, minimality gives $f(a)=a$ and $f(b)=b$.
Isotonicity forces $f(c)$ and $f(d)$ to be above both $a$ and $b$;
the pointwise bounds $f(c)\leq c$ and $f(d)\leq d$ then give
$f(c)=c$ and $f(d)=d$. Thus $f=\operatorname{id}_P$.
The argument for $f\geq\operatorname{id}_P$ is dual.

The constant maps form a~connected copy of $P$, whereas the identity
map is isolated and is not constant. Hence
$\operatorname{id}_P\notin\Conn(P,P)$ and
$\Conn(P,P)\neq P^P$. This ordered set is the four-element crown;
see also the diagram of $P^P$ in \cite[Figure~1]{FarBC}.

We will use the following identities, valid for ordered sets
$P$, $Q$, $Y$, and $Z$:
\begin{align}
   \Conn(P\times Q,Y)
      &\cong \Conn(P,Y)\times\Conn(Q,Y),\label{E:C-product}\\
   \Conn(P,Y+Z)
      &\cong \Conn(P,Y)\times\Conn(P,Z),\label{E:C-sum}\\
   \Conn\bigl(\Conn(P,Y),Z\bigr)
      &\cong \Conn(P,Y\times Z).\label{E:C-iteration}
\end{align}
See \cite[Theorem~2.8]{FarBBb} and
\cite[Proposition~1]{FarBC}.
Here are direct proofs for finite ordered sets.

\begin{proof}[Proof of \eqref{E:C-product}--\eqref{E:C-iteration}]
The natural isomorphism $(P\times Q)^Y\cong P^Y\times Q^Y$
sends a~map constant on each component of $Y$ to a~pair of maps
with the same property, and conversely. Since the components of
a~product are the products of components, it restricts to the
isomorphism in \eqref{E:C-product}.
Similarly, restriction gives $P^{Y+Z}\cong P^Y\times P^Z$ and
identifies the selected components, proving \eqref{E:C-sum}.

For \eqref{E:C-iteration}, use the \emph{currying isomorphism}
\begin{align*}
   P^{Y\times Z}&\longrightarrow(P^Y)^Z,\\
   f&\longmapsto\bigl(z\mapsto(y\mapsto f(y,z))\bigr).
\end{align*}
Its inverse sends $h\colon Z\to P^Y$ to the map
$(y,z)\mapsto h(z)(y)$. These correspondences preserve and reflect
the pointwise order. We show that they restrict to the ordered sets
in \eqref{E:C-iteration}.

Let $f$ lie in $\Conn(P,Y\times Z)$, and choose a~comparability
path from $f$ to a~map $g$ constant on every component of
$Y\times Z$. These components are the products of components
of $Y$ and $Z$. Thus, for each $z\in Z$, the map $y\mapsto g(y,z)$
is constant on each component of $Y$. Moreover, the
curried map $Z\to P^Y$, defined by
$z\mapsto (y\mapsto g(y,z))$, is constant on each
component of~$Z$.
Fix an arbitrary $z\in Z$. Evaluating along the curried path gives
a~path in $P^Y$ ending at the map $y\mapsto g(y,z)$, which is
constant on each component of $Y$. Every map along this evaluated
path therefore belongs to $\Conn(P,Y)$.

Since $z$ was arbitrary, every map along the curried path takes all
its values in $\Conn(P,Y)$. Hence the curried path lies in
$\Conn(P,Y)^Z$ and ends at a~map constant on each component of $Z$.
Thus the curried map of $f$ belongs to $\Conn(\Conn(P,Y),Z)$.

Conversely, let $h$ be an element of the ordered set
$\Conn(\Conn(P,Y),Z)$.
Choose a~path in $\Conn(P,Y)^Z$ from $h$ to a~map that is constant
on each component $Z_j$ of $Z$, with value $h_j\in\Conn(P,Y)$.
For each $j$, choose a~path in $P^Y$ from $h_j$ to a~map $g_j$
constant on every component of $Y$. This path lies in
$\Conn(P,Y)$. Change the value on one $Z_j$ at a~time along its
chosen path, keeping it constant throughout $Z_j$.
Each resulting map on $Z$ is isotone, since distinct components
have no comparabilities. Joining these finitely many paths gives
a~path ending at the map with value $g_j$ on $Z_j$. After uncurrying, its
endpoint is constant on every component of $Y\times Z$.
Thus the uncurried map of $h$ belongs to $\Conn(P,Y\times Z)$.
The currying isomorphism therefore restricts to
\eqref{E:C-iteration}.
\end{proof}

The height of $P^Y$ is easy to compute. The same formula holds for
$\Conn(P,Y)$; see also \cite[Lemma~15(4)]{FarBC}.

\begin{lemma}\label{L:height}
For ordered sets $P$ and $Y$,
\begin{equation}\label{E:height-exponent}
   \height\bigl(\Conn(P,Y)\bigr)=\height(P^Y)=|Y|\height(P).
\end{equation}
\end{lemma}

\begin{proof}
Along a~chain in $P^Y$, the value at a~fixed element of $Y$ can
increase strictly at most $\height(P)$ times.
Every strict step changes at least one value.
This gives the upper bound $|Y|\height(P)$.

To attain this bound, choose a~chain $p_0<\cdots<p_h$ in $P$, where
$h=\height(P)$, and a~linear extension of $Y$.
Starting with the constant map $p_0$, change the values from $p_{j-1}$
to $p_j$ one at a~time, in reverse order of this linear extension,
for $j=1,\dots,h$.
During each transition, the elements already changed form an upper
set of $Y$. Since $p_{j-1}<p_j$, each intermediate map is isotone.
The resulting chain has length $|Y|h$. It starts at a~constant map,
so it lies in $\Conn(P,Y)$.
\end{proof}

\subsection{A~normal form}

A~connected, directly indecomposable ordered set~$E$ is
\emph{absolutely $\Conn$-indecomposable} if a~representation
$E\cong\Conn(P,Y)$ implies that $|Y|=1$ and $P\cong E$.

The following normal form follows from McKenzie's structure
theory \cite[Theorem~5.1]{McKJJ}. Its uniqueness and its behavior
under $\Conn$ are stated in \cite[Lemma~9(3)]{FarBC}; see also
Theorem~2(4), Theorem~4, and Lemma~8 there.
We recall the short argument that gives the form needed below.

\begin{lemma}[Normal form]\label{L:normal}
Let $D$ be a~connected, nontrivial, directly indecomposable ordered set.
There are ordered sets $E_D$ and $F_D$, unique up to isomorphism, such that
$E_D$ is absolutely $\Conn$-indecomposable, $F_D$ is connected,
and
\begin{equation}\label{E:normal}
   D\cong\Conn(E_D,F_D).
\end{equation}
For a~connected ordered set $Y$, the ordered set $\Conn(D,Y)$
is connected, nontrivial, and directly indecomposable, and
\begin{align}
   E_{\Conn(D,Y)}&\cong E_D,\label{E:normal-exponent}\\
   F_{\Conn(D,Y)}&\cong F_D\times Y.\label{E:normal-exponent1}
\end{align}
\end{lemma}

\begin{proof}
Existence follows from \cite[Theorem~2(4)]{FarBC}, which restates
McKenzie's result. Lemma~8 of the same paper shows that $F_D$ is connected.

For uniqueness, suppose that also
$D\cong\Conn(E',F')$, where $E'$ is absolutely $\Conn$-indecomposable
and $F'$ is connected.
By \cite[Theorem~4]{FarBC}, there are connected ordered sets
$H$, $U$, $V$, and $Z$ such that
\begin{align*}
 E_D&\cong\Conn(H,U),\\
 E'&\cong\Conn(H,V),\\
 F_D&\cong V\times Z,\\
 F'&\cong U\times Z.
\end{align*}

Absolute $\Conn$-indecomposability gives $U\cong V\cong\one$ and
$E_D\cong E'\cong H$. Hence $F_D\cong F'\cong Z$.

For connected $Y$, the connectedness and direct
indecomposability of $\Conn(D,Y)$ follow from
\cite[Lemma~9(1)]{FarBC}. Constant maps embed $D$ into $\Conn(D,Y)$,
so this ordered set is nontrivial.
Finally, \eqref{E:C-iteration} gives
\[
   \Conn(D,Y)
      \cong\Conn\bigl(\Conn(E_D,F_D),Y\bigr)
      \cong\Conn(E_D,F_D\times Y).
\]
Since $F_D\times Y$ is connected, uniqueness gives
\eqref{E:normal-exponent} and \eqref{E:normal-exponent1}.
\end{proof}

\section{The construction \texorpdfstring{$\Tow$}{Tow}}\label{S:Lambda}

In the normal form $D\cong\Conn(E_D,F_D)$, applying $\Conn(-,Y)$
with connected $Y$ leaves $E_D$ unchanged and replaces $F_D$ by
$F_D\times Y$. We therefore keep the exponents $F_D$ and collect
them in a~disjoint sum. This will expose~$Y$ as a~product factor.
The construction forgets the bases $E_D$ and retains the exponent
ordered sets $F_D$, with their multiplicities. It also turns direct
products into sums and lowers height.

\begin{definition}\label{D:Lambda}
Let $P$ be a~connected ordered set, with the direct
decomposition \eqref{E:direct}.
Define
\begin{equation}\label{E:Lambda}
   \Tow(P)=F_{D_1}+\cdots+F_{D_k}.
\end{equation}
For $P=\one$, this is the empty sum: $\Tow(\one)=\emptyset$.
\end{definition}

Hashimoto's Refinement Theorem, together with Lemma~\ref{L:normal},
shows that $\Tow(P)$ is well-defined up to isomorphism.

\smallskip
\noindent\emph{Example (Chains).}
For $n\geq2$, an isotone map from $C_{n-1}$ to $C_2$ is determined by
the cut between its values $0$ and $1$. These $n$ maps form a~chain,
so
\[
   C_n\cong C_2^{C_{n-1}}=\Conn(C_2,C_{n-1}).
\]
The chain $C_2$ is absolutely $\Conn$-indecomposable. Indeed, a~representation
$C_2\cong\Conn(P,Y)$ gives $1=|Y|\height(P)$ by Lemma~\ref{L:height}.
Thus $Y\cong\one$ and $P\cong C_2$.

A~product of two nontrivial ordered sets contains incomparable
elements, so every nontrivial chain is directly indecomposable.
Lemma~\ref{L:normal} identifies its normal-form pair as
$(C_2,C_{n-1})$. Hence
\begin{equation}\label{E:Tow-chain}
   \Tow(C_n)\cong C_{n-1}\qquad(n\geq2).
\end{equation}
For example,
\[
   \Tow(C_3\times C_4)\cong C_2+C_3,
   \qquad
   \Tow(C_3\times C_3)\cong C_2+C_2.
\]
Thus $\Tow$ can turn a~connected ordered set into a~disconnected one.
The second example shows why repeated factors contribute repeated
summands.

\smallskip
\noindent\emph{Example (The bases are not retained).}
Let $V$ be the three-element ordered set with one least element and
two incomparable maximal elements. Its prime cardinality makes it
directly indecomposable. A~representation $V\cong\Conn(P,Y)$ gives
\[
   1=\height(V)=|Y|\height(P),
\]
so $Y\cong\one$ and $P\cong V$. Thus $V$ is absolutely
$\Conn$-indecomposable and
\[
   \Tow(V)\cong\one\cong\Tow(C_2),
\]
although $V\ncong C_2$. The construction $\Tow$ alone does not
determine its input. Reconstruction will use its identities together
with $\Max$ and cancellation.

The next lemma gives the two identities used in reconstruction.
It also bounds the cardinality of $\Tow(P)$, from which the strict
decrease of height follows.

\begin{lemma}\label{L:Lambda}
Let $P$ and $Q$ be connected ordered sets, and let
$Y$ be an ordered set.
Then
\begin{align}
   \Tow(P\times Q)
      &\cong\Tow(P)+\Tow(Q),\label{E:Lambda-product}\\
   \Tow\bigl(\Conn(P,Y)\bigr)
      &\cong Y\times\Tow(P).\label{E:Lambda-exponent}
\end{align}
For every connected $P$,
\begin{equation}\label{E:Lambda-size}
   |\Tow(P)|\leq\height(P).
\end{equation}
If $P$ is nontrivial, then $\Tow(P)$ is nonempty and
\begin{equation}\label{E:Lambda-height}
   \height\bigl(\Tow(P)\bigr)<\height(P).
\end{equation}
\end{lemma}

\begin{proof}
The factors of $P\times Q$ are the factors of $P$ together with those
of $Q$, proving \eqref{E:Lambda-product}.

Write $P\cong\prod_{j=1}^{k}D_j$ as in \eqref{E:direct}, and let
$Y=Y_1+\cdots+Y_r$ be the decomposition of $Y$ into connected components.
By \eqref{E:C-product} and \eqref{E:C-sum},
\[
   \Conn(P,Y)\cong
      \prod_{i=1}^{r}\prod_{j=1}^{k}\Conn(D_j,Y_i).
\]
Each factor is connected and directly indecomposable by
Lemma~\ref{L:normal}.
The same lemma gives
\[
   \Tow\bigl(\Conn(P,Y)\bigr)
      \cong\sum_{i=1}^{r}\sum_{j=1}^{k}(F_{D_j}\times Y_i)
      \cong Y\times\Tow(P),
\]
proving \eqref{E:Lambda-exponent}.
The empty product and empty sum cover the case $P=\one$.

For $P=\one$, inequality~\eqref{E:Lambda-size} reads $0\leq0$.
Now let $P$ be nontrivial, so $k\geq1$ and $\Tow(P)$ is nonempty.
For each factor $D_j$, the ordered set $E_{D_j}$ is connected and
nontrivial, and therefore $\height(E_{D_j})\geq1$.
By Lemma~\ref{L:height},
\[
   |F_{D_j}|\leq |F_{D_j}|\height(E_{D_j})=\height(D_j).
\]
Summing over the factors and using \eqref{E:height-product}, we get
\[
   |\Tow(P)|=\sum_{j=1}^{k}|F_{D_j}|
      \leq\sum_{j=1}^{k}\height(D_j)=\height(P).
\]
This proves \eqref{E:Lambda-size}. Since $\Tow(P)$ is nonempty,
\[
   \height(\Tow(P))\leq|\Tow(P)|-1<\height(P),
\]
which proves \eqref{E:Lambda-height}. No height is assigned to the
empty sum $\Tow(\one)$.
\end{proof}

\section{Ordering the connected ordered sets}\label{S:ordering}

An ordered set can have nonisomorphic intervals of maximum height.
For example, $C_3+(C_2\times C_2)$ has two such intervals, as shown
in Figure~\ref{F:interval-choice}. We need a~rule for choosing between
them that remains compatible with products and exponentiation.
An arbitrary choice provides no such compatibility.

\begin{figure}[ht]
\centering
\setlength{\unitlength}{1pt}
\begin{picture}(180,85)
  \put(35,20){\line(0,1){50}}
  \put(35,20){\circle*{3}}
  \put(35,45){\circle*{3}}
  \put(35,70){\circle*{3}}
  \put(115,20){\line(-1,1){25}}
  \put(115,20){\line(1,1){25}}
  \put(90,45){\line(1,1){25}}
  \put(140,45){\line(-1,1){25}}
  \put(115,20){\circle*{3}}
  \put(90,45){\circle*{3}}
  \put(140,45){\circle*{3}}
  \put(115,70){\circle*{3}}
  \put(35,3){\makebox(0,0){$C_3$}}
  \put(115,3){\makebox(0,0){$C_2\times C_2$}}
\end{picture}
\caption{Two nonisomorphic intervals of maximum height in
$C_3+(C_2\times C_2)$.}\label{F:interval-choice}
\end{figure}

We construct an ordering that respects products and the operation
$\Conn$. We order isomorphism types. This gives a~relation $\preceq$ on
connected ordered sets for which
\[
   P\preceq Q\text{ and }Q\preceq P
      \quad\text{if and only if}\quad P\cong Q.
\]
We write $P\prec Q$ when $P\preceq Q$ and $P\ncong Q$.

\begin{theorem}\label{T:ordering}
There is a~linear ordering $\preceq$ of the isomorphism types of
connected ordered sets with the following properties.
For connected ordered sets $P$, $Q$, and $R$, and
an ordered set $Y$,
\begin{align*}
   P\prec Q &\quad\text{implies that}\quad
      P\times R\prec Q\times R,\\
   P\prec Q &\quad\text{implies that}\quad
      \Conn(P,Y)\prec\Conn(Q,Y).
\end{align*}
Both implications also hold with $\preceq$ in place of $\prec$.
\end{theorem}

We construct the ordering in three steps. First we order connected
ordered sets by their factor multiplicities, obtaining compatibility
with products. We then order directly indecomposable factors by their
normal-form pairs, obtaining compatibility with connected exponents.
Finally, we compare factor multiplicities using this second ordering.

The first and last steps compare functions at their largest point of
disagreement. Such a~point exists because each comparison involves
only finitely many differences. No well-ordering is required.
We begin with this observation.

\begin{lemma}\label{L:lex}
Let $C$ and $N$ be chains, possibly infinite, and let $\F$ be a~set of functions
from $C$ to $N$ such that any two functions differ at only finitely
many elements of $C$.
For distinct $f,g\in\F$, let $\did(f,g)$ be the largest element
at which they differ.
Set $f<g$ if $f(\did(f,g))<g(\did(f,g))$.
Then $<$ is a~linear ordering of $\F$.
\end{lemma}

\begin{proof}
Comparability and asymmetry follow from the ordering of $N$.
For transitivity, let $f<g<h$ and put
\[
   c=\max\{\did(f,g),\did(g,h)\}.
\]
The three functions agree above $c$.
At $c$, we have $f(c)\leq g(c)\leq h(c)$, with at least one strict
inequality.
Hence $\did(f,h)=c$ and $f<h$.
\end{proof}

\subsection{A~preliminary ordering}

There are countably many finite ordered sets up to isomorphism.
Enumerate
\[
   \DirInd=\{D_1,D_2,\dots\}.
\]
For a~connected ordered set $P$, write
\begin{align*}
   P&\cong\prod_{i\geq1}D_i^{\times k_i},\\
   \vec P&=(k_1,k_2,\dots).
\end{align*}
Only finitely many $k_i$ are nonzero. The product therefore has
only finitely many nontrivial factors.
By Hashimoto's Refinement Theorem, $\vec P$ is unique.

Compare the vectors at their last unequal coordinate, with the indices
in their usual order. Write $P\leq_1 Q$ for this comparison.
By Lemma~\ref{L:lex}, with $C=\N$ and $N=\Nzero$, it linearly orders
the isomorphism types.
Since
\begin{equation}\label{E:vectors}
   \overrightarrow{P\times R}=\vec P+\vec R,
\end{equation}
multiplication adds the same vector to both sides of a~comparison.
The last unequal coordinate and the inequality there remain unchanged.
Thus multiplication by a~connected ordered set preserves
both $\leq_1$ and its strict form $<_1$.

\subsection{Ordering the directly indecomposable factors}

For $D\in\DirInd$, use the normal form $D\cong\Conn(E_D,F_D)$.
Compare the pairs $(E_D,F_D)$ lexicographically using $\leq_1$:
compare the first coordinates, and then the second coordinates if the
first are isomorphic.
Denote this ordering of $\DirInd$ by $\leq_2$.
It is linear, since $(E_D,F_D)$ determines $D$ up to isomorphism.

Let $Y$ be connected. We need two facts: distinct factors must
remain distinct under $\Conn(-,Y)$, and their relative order must
be preserved. The first keeps their multiplicities separate; the
second preserves the largest point of disagreement.
By Lemma~\ref{L:normal}, the normal-form pair of $\Conn(D,Y)$ is
\[
   (E_D,F_D\times Y).
\]
The first coordinate is unchanged. Multiplication by $Y$ preserves
strict comparisons of the second coordinates. Therefore,
\begin{equation}\label{E:prime-order}
   D<_2 D'
      \quad\text{implies that}\quad
   \Conn(D,Y)<_2\Conn(D',Y).
\end{equation}
Distinct factors also remain distinct. If $\Conn(D,Y)\cong\Conn(D',Y)$,
uniqueness of the normal form gives
\begin{align*}
   E_D&\cong E_{D'},\\
    F_D\times Y&\cong F_{D'}\times Y.
\end{align*}
By Lemma~\ref{L:cancellation}, $F_D\cong F_{D'}$, so $D\cong D'$.
Thus $D\mapsto\Conn(D,Y)$ is injective on isomorphism types and,
by \eqref{E:prime-order}, preserves their ordering.

\subsection{The required ordering}

Let $P$ and $Q$ be nonisomorphic connected ordered sets.
Their factor multiplicities differ at finitely many members of $\DirInd$.
Let $D$ be the largest of these in the ordering $\leq_2$.
Set $P\prec Q$ if $D$ occurs fewer times in $P$ than in $Q$.
Include equality of isomorphism types to obtain $\preceq$.

\begin{proof}[Proof of Theorem~\ref{T:ordering}]
The multiplicity functions on $(\DirInd,\leq_2)$ satisfy the
hypotheses of Lemma~\ref{L:lex}. Thus $\preceq$ linearly orders the
isomorphism types. Multiplication preserves strict comparisons by
\eqref{E:vectors}.

Suppose that $Y$ is connected.
By \eqref{E:C-product}, the factors of $\Conn(P,Y)$ are obtained by
replacing each factor $D$ of $P$ by $\Conn(D,Y)$, with the same
multiplicity.
The new factors are directly indecomposable by Lemma~\ref{L:normal};
they remain distinct and in the same order by \eqref{E:prime-order}.
The largest factor with unequal multiplicities is therefore replaced
by the largest such factor in the new products. Its multiplicities do
not change. Hence $P\prec Q$ implies $\Conn(P,Y)\prec\Conn(Q,Y)$.

For arbitrary $Y$, write $Y=Y_1+\cdots+Y_r$ as a~sum of connected
components. If $P\prec Q$, the connected case gives
$\Conn(P,Y_i)\prec\Conn(Q,Y_i)$ for every $i$.
We compare the products by changing one factor at a~time.
For $0\leq j\leq r$, put
\[
   \Z_j=\prod_{i=1}^{j}\Conn(Q,Y_i)
         \times\prod_{i=j+1}^{r}\Conn(P,Y_i),
\]
with an empty product interpreted as $\one$.
The unchanged factors have a~connected product. Compatibility with
multiplication gives $\Z_{j-1}\prec \Z_j$ for $1\leq j\leq r$.
Thus $\Z_0\prec \Z_r$, and \eqref{E:C-sum} gives
$\Conn(P,Y)\prec\Conn(Q,Y)$.
Both constructions preserve isomorphisms, so the non-strict
implications follow as well.
\end{proof}

The ordering depends on our enumeration of $\DirInd$.
We fix this enumeration for the rest of the paper. Any such choice
will serve. Intermediate selected intervals may depend on it, but
the reconstruction will always recover the original isomorphism type.

\section{Selecting an interval of maximum height}\label{S:intervals}

An ordered set may have several nonisomorphic intervals of maximum
height. We use the ordering just constructed to choose one.
For an ordered set $X$, let
\[
   \MaxInt(X)=\{[a,b]\subseteq X\mid \height([a,b])=\height(X)\}.
\]
The endpoints of a~chain of maximum length in $X$ determine a~member
of $\MaxInt(X)$. All members of $\MaxInt(X)$ are connected.

\begin{definition}\label{D:T}
Let $\Max(X)$ be a~$\preceq$-greatest member of $\MaxInt(X)$.
\end{definition}

Since $\MaxInt(X)$ is finite, such a choice exists.
It is unique up to isomorphism because $\preceq$
linearly orders the isomorphism types.
If $X$ is bounded, then $\Max(X)\cong X$:
an interval of full height must contain both bounds, since otherwise
a~longest chain in that interval could be extended in $X$.
In particular,
\begin{equation}\label{E:T-invariance}
\begin{gathered}
   X\cong X'\quad\text{implies that}\quad \Max(X)\cong \Max(X'),\\
   \height(\Max(X))=\height(X).
\end{gathered}
\end{equation}

\smallskip
\noindent\emph{The strict-height selection principle.}
For ordered sets $X$ and $Y$,
\begin{equation}\label{E:T-sum-selection}
   \height(X)<\height(Y)
      \quad\text{implies that}\quad
   \Max(X+Y)\cong\Max(Y).
\end{equation}
Every interval in a~disjoint sum lies in one summand. Thus $X$
contributes no interval of maximum height, and the choice is made
entirely among the intervals of $Y$. An empty summand is simply
omitted. Also, repeating a~summand does not change the selected
isomorphism type: $\Max(kY)\cong\Max(Y)$ for $k\geq1$.
We will use these observations in both reconstruction arguments.

\smallskip
\noindent\emph{Example (Which interval is selected?).}
Return to $X=C_3+(C_2\times C_2)$ in
Figure~\ref{F:interval-choice}. The normal-form pairs of $C_2$ and
$C_3$ are $(C_2,\one)$ and $(C_2,C_2)$, respectively.
The multiplicity vector of $\one$ is zero, whereas that of $C_2$
has one nonzero entry. Hence $\one<_1 C_2$, and therefore
$C_2<_2 C_3$.

In the final comparison between $C_2\times C_2$ and $C_3$, the
largest factor with unequal multiplicities is $C_3$. It occurs zero
times in the first ordered set and once in the second. Consequently,
\begin{align*}
   C_2\times C_2\prec C_3,\\
   \Max\bigl(C_3+(C_2\times C_2)\bigr)\cong C_3.
\end{align*}
This comparison does not depend on the initial enumeration of
$\DirInd$. It also shows that $\preceq$ is not an ordering by
cardinality: the selected interval has fewer elements.

To see how this choice behaves under exponentiation, we first describe
the intervals of maximum height in an exponential ordered set;
compare \cite[Lemma~15(3)]{FarBC}.

\begin{lemma}\label{L:exponent-intervals}
Let $X$ and $Y$ be ordered sets, and let
$Y=Y_1+\cdots+Y_r$ be the decomposition of $Y$ into connected components.
The members of $\MaxInt(X^Y)$ are, up to isomorphism, precisely the
ordered sets
\begin{equation}\label{E:exponent-intervals}
   I_1^{Y_1}\times\cdots\times I_r^{Y_r},
\end{equation}
where $I_i\in\MaxInt(X)$ for $1\leq i\leq r$.
\end{lemma}

\begin{proof}
Let $[f,g]\in\MaxInt(X^Y)$ and put $h=\height(X)$.
Bound the height at each coordinate as in the proof of
Lemma~\ref{L:height}:
\[
   |Y|h=\height([f,g])
      \leq\sum_{y\in Y}\height([f(y),g(y)])
      \leq |Y|h.
\]
Equality holds throughout, so every interval $[f(y),g(y)]$ has height
$h$. A~chain of maximum length in this interval contains both endpoints.
Its bottom must be minimal in $X$, and its top must be maximal in $X$;
otherwise we could extend the chain.
If $y,z\in Y$ are comparable, then $f(y)$ and $f(z)$ are comparable
minimal elements, and so are equal. Similarly, $g(y)=g(z)$.
Connectedness now shows that $f$ and $g$ are constant on each $Y_i$,
with values $a_i$ and $b_i$, respectively.
Put $I_i=[a_i,b_i]$. Restriction to the components $Y_i$ gives
an isomorphism
\[
   [f,g]\cong\prod_{i=1}^{r}I_i^{Y_i}.
\]
This proves \eqref{E:exponent-intervals}.

Conversely, choose $I_i=[a_i,b_i]\in\MaxInt(X)$ for every $i$.
Take the maps whose values on $Y_i$ are constantly $a_i$ and $b_i$.
The interval between them has the form \eqref{E:exponent-intervals}.
By Lemma~\ref{L:height} and \eqref{E:height-product}, its height is
$\sum_{i=1}^{r}|Y_i|h=|Y|h$.
\end{proof}

We can now show that $\Max$ respects products and exponentiation.
These two properties are needed for reconstruction.

\begin{lemma}\label{L:T}
For ordered sets $X$ and $Y$,
\begin{align}
   \Max(X\times Y)&\cong \Max(X)\times \Max(Y),\label{E:T-product}\\
   \Max(X^Y)&\cong \Max(X)^Y.\label{E:T-exponent}
\end{align}
\end{lemma}

\begin{proof}
An interval in $X\times Y$ is a~product $I\times J$ of intervals in
$X$ and $Y$. By \eqref{E:height-product}, it has maximum height exactly
when $I\in\MaxInt(X)$ and $J\in\MaxInt(Y)$.
Theorem~\ref{T:ordering} then gives
\[
   I\times J\preceq \Max(X)\times \Max(Y).
\]
The right-hand side is an interval of maximum height, so
\eqref{E:T-product} follows.

For \eqref{E:T-exponent}, write $Y=Y_1+\cdots+Y_r$ as above.
By Lemma~\ref{L:exponent-intervals}, every interval of maximum height
in $X^Y$ is isomorphic to one of the products in
\eqref{E:exponent-intervals}.
The intervals $I_i$ and $\Max(X)$ are bounded, so \eqref{E:bounded}
and Theorem~\ref{T:ordering} give
\[
   I_i^{Y_i}=\Conn(I_i,Y_i)
      \preceq\Conn(\Max(X),Y_i)=\Max(X)^{Y_i}.
\]
Taking products, we obtain the upper bound
\[
   \prod_{i=1}^{r}\Max(X)^{Y_i}\cong\Max(X)^Y
\]
in the ordering $\preceq$ for every interval of maximum height in
$X^Y$. This bound is attained by taking $I_i=\Max(X)$ for every $i$.
Thus \eqref{E:T-exponent} follows.
\end{proof}

\section{Proof of the main theorem}\label{S:reconstruction}

We now recover $A$ from $A^A$. The two constructions have different
roles. By Lemma~\ref{L:T}, $\Max$ preserves products and exponentiation.
By Lemma~\ref{L:Lambda}, $\Tow$ turns products of connected ordered
sets into sums and, for bounded ordered sets, exponentiation into
products. It also lowers the height
of every nontrivial connected ordered set.

First assume that $A$ has positive height. Starting with the ordered
set $E\cong A^A$, we recover $\Max(A)$ and then $A$.
The ordered sets $U,V,W$ in the next proposition are constructed
from $E$ alone. The proof introduces $L_0,R,S$ in terms of $A$ only
to verify the identities; these are not inputs to the construction.

\begin{proposition}\label{P:reconstruction}
Let $A$ be an ordered set of positive height, and let
$E\cong A^A$. Then
\begin{align*}
   U&=\Tow(\Max(E)),\\
   V&=\Max(U),\\
   W&=\Max(\Tow(V))
\end{align*}
are well-defined.
There is a~connected ordered set $L$, unique up to
isomorphism, such that
\begin{equation}\label{E:recover-L}
   V\cong W\times L.
\end{equation}
Moreover, $L\cong \Max(A)$ and
\begin{equation}\label{E:recover-A}
   U\cong\Tow(L)\times A.
\end{equation}
Equation~\eqref{E:recover-A} determines $A$ up to isomorphism.
\end{proposition}

\begin{proof}
Put $L_0=\Max(A)$ and $R=\Tow(L_0)$.
The interval $L_0$ is nontrivial, since $\height(A)\geq1$.
By Lemma~\ref{L:Lambda}, $R$ is nonempty. We may therefore put
$S=\Max(R)$.

By Lemma~\ref{L:T} and the boundedness of $L_0$,
\[
   \Max(E)\cong L_0^A=\Conn(L_0,A).
\]
Apply $\Tow$. By \eqref{E:Lambda-exponent},
\begin{equation}\label{E:U}
   U\cong R\times A.
\end{equation}
The product in \eqref{E:U} is nonempty, so $V=\Max(U)$ is defined.
Applying $\Max$ to \eqref{E:U}, we get
\begin{equation}\label{E:V}
   V\cong S\times L_0.
\end{equation}
Both factors are connected, and $L_0$ is nontrivial. Thus $V$ is
connected and nontrivial, so $\Tow(V)$ is defined and nonempty.
Now \eqref{E:Lambda-product} gives
\begin{equation}\label{E:Lambda-V}
   \Tow(V)\cong\Tow(S)+R.
\end{equation}
Thus $W=\Max(\Tow(V))$ is defined. We claim that $W\cong S$.

If $S\cong\one$, then $\Tow(S)=\emptyset$.
Otherwise, \eqref{E:Lambda-height} gives
\[
   \height(\Tow(S))<\height(S)=\height(R).
\]
By the strict-height selection principle \eqref{E:T-sum-selection},
the summand $\Tow(S)$ contributes no interval to the choice:
it is either empty or has smaller height than $R$. Hence
\[
   W=\Max(\Tow(V))\cong \Max(R)=S.
\]
Equation~\eqref{E:V} now gives $V\cong W\times L_0$.
By cancellation, any $L$ satisfying \eqref{E:recover-L} is isomorphic
to $L_0$.

Thus $L$ is connected and $\Tow(L)\cong R$.
Substituting in \eqref{E:U}, we obtain \eqref{E:recover-A}.
Cancelling $\Tow(L)$ proves that this equation determines $A$.
\end{proof}

The reconstruction requires two cancellations. The known ordered
sets $V$ and $W$ first determine $L$ from $V\cong W\times L$.
Only then do we form $\Tow(L)$ and use it to cancel from
$U\cong\Tow(L)\times A$. Thus neither cancellation assumes prior
knowledge of the factor it is intended to recover.
Each step determines an isomorphism type, using the fixed ordering
from Section~\ref{S:ordering}.

\begin{proof}[Proof of Theorem~\ref{T:main}]
First suppose that $A$ is an antichain with $n$ elements.
Then $A^A$ is an antichain with $n^n$ elements.
The constant maps embed $B$ into $B^B\cong A^A$, so $B$ is also
an antichain. If $|B|=m$, then $m^m=n^n$.
Since $n\mapsto n^n$ is strictly increasing on the positive integers,
we get $m=n$ and $A\cong B$.

Now suppose that $\height(A)\geq1$, and put $E=A^A\cong B^B$.
By Lemma~\ref{L:height}, $B$ also has positive height.
Apply Proposition~\ref{P:reconstruction} to $E$, first with $A$ and
then with $B$. Since $E$ determines $U$, $V$, $W$, and $L$ up to
isomorphism, we obtain that
\[
   \Tow(L)\times A\cong U\cong\Tow(L)\times B.
\]
Cancellation gives $A\cong B$.
\end{proof}

\subsection{Examples of reconstruction}

\ 

\noindent\emph{Height one.}
If $\height(A)=1$, then every interval of maximum height in $A$ is
isomorphic to $C_2$. Since $\Tow(C_2)\cong\one$, equation~\eqref{E:U}
becomes
\begin{equation}\label{E:height-one-recovery}
   A\cong\Tow(\Max(A^A)).
\end{equation}
There is no nontrivial factor to cancel in this case.

For a~disconnected example, take $A=C_2+\one$. An isotone map from
the connected ordered set $C_2$ into $A$ has its image in one
component, so $A^{C_2}\cong C_3+\one$. Restriction to the two
components of the domain now gives
\[
   A^A\cong (C_3+\one)\times(C_2+\one).
\]
Its four components have heights $3$, $2$, $1$, and~$0$.
The first component, $C_3\times C_2$, is bounded, so it is the unique
interval of maximum height. Consequently,
\[
   \Tow(\Max(A^A))
      \cong\Tow(C_3\times C_2)
      \cong C_2+\one=A.
\]
The isolated component of $A$ is recovered even though $\Max(A^A)$
is connected.

\smallskip
\noindent\emph{The three-element chain.}
Let $A=C_3$ and $E=A^A$. Since $E$ is bounded, $\Max(E)\cong E$.
Using $\Tow(C_3)\cong C_2$ and Lemma~\ref{L:Lambda}, we obtain
\[
   U=\Tow(\Max(E))\cong C_3\times C_2.
\]
This product is bounded, so $V=\Max(U)\cong C_3\times C_2$.
Next,
\begin{align*}
   \Tow(V)&\cong C_2+\one,\\
    W=\Max(\Tow(V))&\cong C_2.
\end{align*}
The equation $V\cong W\times L$ therefore recovers $L\cong C_3$.
We now know $\Tow(L)\cong C_2$, and the equation
$U\cong\Tow(L)\times A$ recovers $A\cong C_3$.
Both cancellations in Proposition~\ref{P:reconstruction} are visible
in this example.

\smallskip
\noindent\emph{A~disconnected ordered set with nontrivial cancellation.}
Take $A=C_3+\one$ and $E=A^A$. Restriction to the components of
the domain gives
\[
   E\cong(C_3^{C_3}+\one)\times(C_3+\one).
\]
There are ten isotone maps from $C_3$ to itself: they are the
nondecreasing triples with entries in $C_3$. Thus $|E|=11\cdot4=44$.
The four components have heights $8,6,2,0$. The component of height
$8$ is bounded, so
\[
   \Max(E)\cong C_3^{C_3}\times C_3.
\]
Using $\Tow(C_3)\cong C_2$, we obtain
\begin{align*}
   U&\cong(C_2\times C_3)+C_2,\\
   V&\cong C_2\times C_3,\\
   \Tow(V)&\cong\one+C_2,\\
   W&\cong C_2.
\end{align*}
The equation $V\cong W\times L$ recovers $L\cong C_3$.
We can now form $\Tow(L)\cong C_2$. The second cancellation uses
\[
   U\cong C_2\times(C_3+\one)
\]
and recovers $A\cong C_3+\one$.

Although $\Max(E)$ is connected, the disconnected exponent has not
been lost: its two components reappear in $U$. This example combines
a~disconnected input with two nontrivial cancellations.

\smallskip
\noindent\emph{A~disconnected factor.}
For the four-element Boolean lattice $A=C_2\times C_2$, we have
$\Tow(A)\cong\one+\one$. In this case the reconstruction gives
\begin{align*}
  U&\cong A+A,\\
  V&\cong A,\\
  W&\cong\one,\\
  L&\cong A.
\end{align*}
The last step cancels the two-element antichain $\Tow(L)$ from
$U\cong\Tow(L)\times A$. This is why Lemma~\ref{L:cancellation}
allows disconnected factors.

\section{Extensions to iterated exponentials}\label{S:iteration}

Theorem~\ref{T:main} settles one reconstruction problem. We now ask
what happens when we continue exponentiating. For example, do
$(A^A)^A$ and $A^{(A^A)}$ both determine $A$?

We prove that they do. We also give a~recursive family of
reconstructing expressions, including all left-associated and
right-associated powers. The parameters of an expression are fixed
in advance. The conclusion compares its values at $A$ and $B$ using
these same parameters.

The applications have different dependencies. Left-associated powers
follow from Theorem~\ref{T:main} and the direct-root lemma below.
Repeated self-exponentiation uses Theorem~\ref{T:main} repeatedly.
The recursive theorem supplies the additional right-associated and
mixed families.

An \emph{exponential term} in $x$ is obtained by the following rule:
$x$ is a~term, and $s^t$ is a~term whenever $s$ and $t$ are terms.
We evaluate a~term at an ordered set $A$ by
\begin{align*}
   x(A)&=A,\\
   (s^t)(A)&=s(A)^{t(A)}.
\end{align*}
\begin{definition}
An exponential term $t$ is \emph{reconstructing} if, for finite
ordered sets $A$ and $B$,
\[
   t(A)\cong t(B)\quad\text{implies that}\quad A\cong B.
\]
\end{definition}

Thus the same term, with the same parenthesization, is evaluated at
$A$ and $B$. The assertion does not compare two different terms.

Different parenthesizations need not give isomorphic ordered sets.
For example, take $A=C_2$.
Then $A^A\cong C_3$. The ordered set $(A^A)^A\cong C_3^{C_2}$ has
six elements, whereas $A^{(A^A)}\cong C_2^{C_3}\cong C_4$ has four.
Both expressions nevertheless determine $A$, when each is considered
separately.

\subsection{Exponential laws and direct roots}

For ordered sets $X$, $Y$, and $Z$, we have
\begin{align}
   (X^Y)^Z&\cong X^{Y\times Z},\label{E:iteration-law}\\
   (X\times Y)^Z&\cong X^Z\times Y^Z.\label{E:iteration-base-product}
\end{align}
For the first isomorphism, send $f\colon Z\to X^Y$ to the map
$(y,z)\mapsto f(z)(y)$. For the second, take the two coordinate maps.
Both correspondences preserve and reflect the pointwise order.
In particular,
\begin{equation}\label{E:iteration-interchange}
   (X^Y)^Z\cong(X^Z)^Y.
\end{equation}

\begin{lemma}[Direct roots]\label{L:iteration-roots}
Let $X$ and $Y$ be finite ordered sets, and let $k\geq1$.
If $X^{\times k}\cong Y^{\times k}$, then $X\cong Y$.
\end{lemma}

\begin{proof}
Use the polynomials from the proof of Lemma~\ref{L:cancellation}.
The hypothesis gives $\p_X^k=\p_Y^k$, and hence
\[
   (\p_X-\p_Y)\sum_{j=0}^{k-1}\p_X^{k-1-j}\p_Y^j=0.
\]
The second factor is nonzero: $\p_X$ and $\p_Y$ are nonzero and have
nonnegative coefficients. Since the polynomial ring is an integral
domain, $\p_X=\p_Y$. Hence $X\cong Y$.
\end{proof}

\subsection{Two auxiliary constructions}

We use two combinations of $\Max$ and $\Tow$. The first makes the
exponent available as a~direct factor. For an ordered set $X$, its
\emph{derived sum} is
\begin{equation}\label{E:iteration-der}
   \Der(X)=\Tow(\Max(X)).
\end{equation}
For an antichain $X$, this is the empty sum.
The second construction lowers height. For a~connected, nontrivial
ordered set $P$, its \emph{height reduction} is
\begin{equation}\label{E:iteration-red}
   \Red(P)=\Max(\Tow(P)).
\end{equation}
By Lemma~\ref{L:Lambda}, $\Red(P)$ is defined and has smaller height
than $P$. With the auxiliary ordering fixed, both constructions
depend only on the isomorphism type of their argument.

\begin{lemma}\label{L:iteration-der}
For ordered sets $X$ and $Y$,
\begin{align}
   \Der(X\times Y)&\cong\Der(X)+\Der(Y),
      \label{E:iteration-der-product}\\
   \Der(X^Y)&\cong\Der(X)\times Y.
      \label{E:iteration-der-exponent}
\end{align}
\end{lemma}

\begin{proof}
Combine \eqref{E:T-product} and \eqref{E:Lambda-product} for the first
identity. For the second, use $\Max(X^Y)\cong\Max(X)^Y$ and the
boundedness of $\Max(X)$. Equations~\eqref{E:bounded} and
\eqref{E:Lambda-exponent} give
\[
   \Tow(\Max(X)^Y)
      =\Tow(\Conn(\Max(X),Y))
      \cong\Tow(\Max(X))\times Y.\qedhere
\]
\end{proof}

\subsection{A~recursive family}

At each stage, the base is $A$. The exponent is a~direct product
of copies of $A$ and copies of the preceding term. We allow zero
copies of $A$, but require at least one copy of the preceding term.

The depth and the numbers of copies are fixed before the ordered set
is chosen. We use the same parameters when comparing $A$ and $B$.

\begin{theorem}\label{T:iteration-reconstruction}
Fix integers $d\geq1$, $r_i\geq0$, and $s_i\geq1$, for
$1\leq i\leq d$. For a~finite ordered set $A$, define
\begin{align}
   E_0(A)&=A,\notag\\
   D_i(A)&=A^{\times r_i}\times E_{i-1}(A)^{\times s_i},\notag\\
   E_i(A)&=A^{D_i(A)}\qquad(1\leq i\leq d).
      \label{E:iteration-family}
\end{align}
For finite ordered sets $A$ and $B$, using the same $d$, $r_i$, and
$s_i$ in both constructions,
\[
   E_d(A)\cong E_d(B)\quad\text{implies that}\quad A\cong B.
\]
\end{theorem}

For example, take $d=1$. Then
\[
   E_1(A)\cong A^{A^{\times(r_1+s_1)}}.
\]
Thus every fixed direct power of $A$ may serve as the exponent.
If $r_i=0$ and $s_i=1$ at every stage, we obtain
\begin{align*}
 E_1(A)&\cong A^A,\\
   E_2(A)&\cong A^{(A^A)},\\
   E_3(A)&\cong A^{(A^{(A^A)})}.
\end{align*}
These are the right-associated powers. The theorem also allows mixed
choices; for instance, $d=2$, $r_1=0$, $s_1=2$, $r_2=s_2=1$ gives
\[
   E_2(A)\cong A^{\,A\times A^{A^{\times2}}}.
\]

Each $E_i(A)$ can be written using exponentiation alone. Start with
$A$ and exponentiate successively by $r_i$ copies of $A$, then by
$s_i$ copies of $E_{i-1}(A)$, always associating to the left.
By \eqref{E:iteration-law}, the resulting term has value isomorphic
to $E_i(A)$. Thus the theorem gives reconstructing exponential terms.

\begin{proof}
Fix $d$, $r_i$, and $s_i$. We recover $A$ from $E_d(A)$ up
to isomorphism. For positive-height $A$, the reconstruction has two
parts. We first recover $\Max(A)$ by repeatedly lowering height.
This determines the auxiliary factor that we must cancel. We then
recover the exponents $D_d(A),\dots,D_1(A)$ in reverse order.
The last of these is a~known direct power of $A$.

First let $A$ be an antichain with $n$ elements.
All $D_i(A)$ and $E_i(A)$ are then antichains. Writing
$e_i(n)=|E_i(A)|$, we have
\[
   e_0(n)=n,\qquad
   e_i(n)=n^{\,n^{r_i}e_{i-1}(n)^{s_i}}.
\]
By induction, each $e_i$ is strictly increasing on the positive
integers: the exponent is positive and nondecreasing, while the base
increases strictly. The constant maps embed $A$ into every $E_i(A)$
for $i\geq1$. Thus $E_d(A)$ is an antichain exactly when $A$ is,
and its size determines $n$. This recovers $A$ in the antichain case.

Now let $\height(A)>0$. Write $E_i=E_i(A)$ and $D_i=D_i(A)$, and put
\begin{align*}
   L&=\Max(A),\\ R&=\Tow(L),\\ S&=\Max(R).
\end{align*}
Since $L$ is nontrivial, Lemma~\ref{L:Lambda} shows that $R$ and $S$
are defined and
\begin{equation}\label{E:iteration-height-drop}
   \height(S)=\height(R)<\height(L).
\end{equation}
Every $E_i$ has positive height. Since $s_i\geq1$, the same is true
of every $D_i$. These strict inequalities will identify the summands
that survive the height reductions.

\smallskip
\noindent\emph{Recovering $L$.}
Put $Q_i=\Max(E_i)$ for $0\leq i\leq d$, so $Q_0=L$.
For $1\leq i\leq d$, \eqref{E:iteration-der-exponent} gives
\begin{equation}\label{E:iteration-Q}
   \Tow(Q_i)=\Der(E_i)\cong R\times D_i.
\end{equation}
Consequently, \eqref{E:T-product} yields
\begin{equation}\label{E:iteration-V}
   V_i =\Red(Q_i)
      \cong S\times L^{\times r_i}\times Q_{i-1}^{\times s_i}.
\end{equation}
Each $V_i$ is therefore connected and nontrivial.
We next show how to pass from $V_i$ to $V_{i-1}$.
For $i\geq2$, \eqref{E:Lambda-product} and \eqref{E:iteration-Q} give
\begin{equation}\label{E:short}
   \Tow(V_i)
      \cong\Tow(S)+r_iR+s_i(R\times D_{i-1}).
\end{equation}
The last summands have height
$\height(R)+\height(D_{i-1})>\height(R)$.
If $S\cong\one$, then $\Tow(S)=\emptyset$; otherwise,
\[
   \height(\Tow(S))<\height(S)=\height(R).
\]
By \eqref{E:T-sum-selection}, the intervals contributing to the
selection therefore lie in the $s_i$ copies of $R\times D_{i-1}$.
These copies have the same isomorphism type, so their number does
not affect the selected interval. Hence
\begin{equation}\label{E:iteration-descent}
   \Red(V_i)\cong\Max(R\times D_{i-1})\cong V_{i-1}
   \qquad(2\leq i\leq d).
\end{equation}

Set $k=r_1+s_1$, so $k\geq1$. Equation~\eqref{E:iteration-V} gives
\begin{align*}
   V_1&\cong S\times L^{\times k},\\
   \Tow(V_1)&\cong\Tow(S)+kR.
\end{align*}
The same comparison of heights, now applied to $\Tow(S)+kR$,
gives $\Red(V_1)\cong\Max(R)\cong S$.
Starting from $Q_d$ and applying \eqref{E:iteration-descent}, we obtain
\begin{equation}\label{E:iteration-recover-L}
\begin{aligned}
   \Red^d(Q_d)&\cong S\times L^{\times k},\\
   \Red^{d+1}(Q_d)&\cong S.
\end{aligned}
\end{equation}
The superscripts on $\Red$ denote composition.
Every application is defined, since its argument is $Q_d$ or one of
the nontrivial ordered sets $V_i$. We stop when we reach $S$.

Both ordered sets on the left of \eqref{E:iteration-recover-L} are
known from $E_d$. Cancellation recovers $L^{\times k}$, and
Lemma~\ref{L:iteration-roots} then recovers $L$.
We also obtain $R=\Tow(L)$.

\smallskip
\noindent\emph{Recovering $A$.}
By \eqref{E:iteration-Q},
\[
   \Der(E_d)\cong R\times D_d.
\]
Cancelling the nonempty ordered set $R$ recovers $D_d$.
We now work backwards from $D_d$ to $D_1$. At this stage, component
multiplicities retain the information that the height reductions
were not intended to preserve.
For $2\leq i\leq d$, Lemma~\ref{L:iteration-der} gives
\begin{align}
   \Der(D_i)
      &\cong r_iR+s_i\Der(E_{i-1})\notag\\
      &\cong r_iR+s_i(R\times D_{i-1})\notag\\
      &\cong R\times(r_i\one+s_iD_{i-1}).
         \label{E:iteration-recover-D}
\end{align}
From $D_i$, cancellation of $R$ determines
$K_i\cong r_i\one+s_iD_{i-1}$. We recover $D_{i-1}$ by counting
components. For a~connected isomorphism type $T$, let $c_T(X)$ be the
number of components of $X$ of type $T$. Then
\[
   c_T(K_i)=
   \begin{cases}
      r_i+s_i c_{\one}(D_{i-1}),&\text{for }T\cong\one,\\
      s_i c_T(D_{i-1}),&\text{for }T\ncong\one.
   \end{cases}
\]
Consequently,
\[
   c_T(D_{i-1})=
   \begin{cases}
      \bigl(c_{\one}(K_i)-r_i\bigr)/s_i,&\text{for }T\cong\one,\\
      c_T(K_i)/s_i,&\text{for }T\ncong\one.
   \end{cases}
\]
Equation~\eqref{E:iteration-recover-D} ensures that these numbers are
nonnegative integers. They determine every component multiplicity
of $D_{i-1}$, and hence $D_{i-1}$ itself up to isomorphism.
For the one-element components, we simply subtract $r_i$ from their
number before dividing by $s_i$.

Repeating this step, we reach
\[
   D_1\cong A^{\times(r_1+s_1)}=A^{\times k}.
\]
Lemma~\ref{L:iteration-roots} recovers $A$.
When $d=1$, we already have $D_1$ after the first cancellation and
can apply the direct-root lemma at once.
\end{proof}

\smallskip
\noindent\emph{Example (Recovering component multiplicities).}
Take $A=C_2+2\one$, $d=2$, $r_1=0$, $s_1=1$, $r_2=1$, and
$s_2=2$. Then $D_1\cong A$, and the recovery of $D_1$ from $D_2$
passes through
\[
   K_2\cong\one+2D_1\cong2C_2+5\one.
\]
The number of $C_2$ components of $D_1$ is $2/2=1$.
For the singleton components, subtract the known contribution
$r_2=1$ before dividing by $s_2=2$:
\[
   c_{\one}(D_1)=\frac{5-1}{2}=2.
\]
Thus $D_1\cong C_2+2\one$. Dividing the five singleton components
by $2$ without first subtracting $1$ would not perform the required
recovery.

\subsection{Parenthesizations covered by the results}

We begin with three familiar ways to repeat exponentiation:
associate to the left, associate to the right, or repeatedly apply
$X\mapsto X^X$.

\begin{corollary}\label{C:iteration-towers}
For a~finite ordered set $A$, define
\begin{align*}
   F_1(A)&=A,& F_{m+1}(A)&=F_m(A)^A,\\
   G_1(A)&=A,& G_{m+1}(A)&=A^{G_m(A)},\\
   H_0(A)&=A,& H_{m+1}(A)&=H_m(A)^{H_m(A)}.
\end{align*}
For every fixed $m\geq1$, each of $F_m(A)$, $G_m(A)$, and $H_m(A)$
determines $A$ up to isomorphism.
\end{corollary}

\begin{proof}
For $m=1$, we have $F_1(A)=G_1(A)=A$.
Let $m\geq2$ and put $k=m-1$. The exponential laws give
\begin{align*}
   F_m(A)&\cong A^{A^{\times k}},\\
   F_m(A)^{\times k}&
      \cong\bigl(A^{\times k}\bigr)^{A^{\times k}}.
\end{align*}
Theorem~\ref{T:main} recovers $A^{\times k}$ from $F_m(A)^{\times k}$,
and Lemma~\ref{L:iteration-roots} recovers $A$.
We could also apply Theorem~\ref{T:iteration-reconstruction} with
$d=1$, $r_1=0$, and $s_1=k$.

For $G_m$ with $m\geq2$, apply Theorem~\ref{T:iteration-reconstruction}
with $d=m-1$, $r_i=0$, and $s_i=1$.
For $H_m$, apply Theorem~\ref{T:main} $m$ times.
\end{proof}

\smallskip
\noindent\emph{Example: right-associated powers of a~two-element chain.}
By the chain calculation in Section~\ref{S:Lambda},
$C_2^{C_{n-1}}\cong C_n$ for $n\geq2$.
Equation~\eqref{E:Tow-chain} and the boundedness of chains give
\[
   \Red(C_n)=\Max(\Tow(C_n))\cong C_{n-1}\qquad(n\geq2).
\]

It follows that $G_m(C_2)\cong C_{m+1}$. For example, take $d=2$,
$r_1=r_2=0$, and $s_1=s_2=1$. Then $E_2(C_2)\cong C_4$, and successive
height reductions give
\[
   C_4\longmapsto C_3\longmapsto C_2\longmapsto\one.
\]
Here $L=C_2$ and $R=S=\one$. The last two ordered sets are
$S\times L$ and $S$, as required by \eqref{E:iteration-recover-L}.
Cancellation recovers $L$; the second part of the proof then recovers
$A$.

Substitution gives further reconstructing terms.

\begin{proposition}[Substitution]\label{P:iteration-substitution}
If $s$ and $t$ are reconstructing exponential terms, then the term
obtained by substituting $t$ for every occurrence of $x$ in $s$ is
also reconstructing.
\end{proposition}

\begin{proof}
The value of the substituted term at $A$ is $s(t(A))$. Therefore,
\[
   s(t(A))\cong s(t(B))
      \quad\text{implies that}\quad t(A)\cong t(B)
      \quad\text{implies that}\quad A\cong B.
\]
The intermediate ordered sets $t(A)$ and $t(B)$ are finite, so both
applications of reconstruction are valid.
\end{proof}

We may also replace a~term by another with isomorphic values on all
finite ordered sets. In particular, \eqref{E:iteration-interchange}
may be used within any reconstructing term.

\begin{corollary}\label{C:iteration-four}
Every exponential term with at most four occurrences of $x$ is
reconstructing.
\end{corollary}

\begin{proof}
The term $x$ is reconstructing, as is $x^x$ by Theorem~\ref{T:main}.
For three occurrences, Corollary~\ref{C:iteration-towers} applies to
both $(x^x)^x$ and $x^{(x^x)}$.
For four occurrences, the five values and their sources of
reconstruction are as follows.
\begin{center}
\renewcommand{\arraystretch}{1.6}
\setlength{\tabcolsep}{12pt}
\begin{tabular}{@{}ll@{}}
\toprule
Value at $A$ & Reconstruction \\
\midrule
$\bigl((A^A)^A\bigr)^A$ & $F_4$: Corollary~\ref{C:iteration-towers} \\
$\bigl(A^{(A^A)}\bigr)^A$ & Isomorphic to $H_2(A)$, as below \\
$(A^A)^{(A^A)}$ & $H_2$: Corollary~\ref{C:iteration-towers} \\
$A^{((A^A)^A)}$ & Theorem~\ref{T:iteration-reconstruction}, as below \\
$A^{(A^{(A^A)})}$ & $G_4$: Corollary~\ref{C:iteration-towers} \\
\bottomrule
\end{tabular}
\end{center}The second value is isomorphic to the third, since
\[
   \bigl(A^{(A^A)}\bigr)^A
      \cong A^{(A^A)\times A}
      \cong A^{A\times(A^A)}
      \cong(A^A)^{(A^A)}.
\]
For the fourth, take $d=2$, $r_1=r_2=0$, $s_1=2$, and $s_2=1$
in Theorem~\ref{T:iteration-reconstruction}. Then
$E_1(A)\cong(A^A)^A$ and $E_2(A)\cong A^{((A^A)^A)}$.
\end{proof}

\section{Further questions}\label{S:questions}

\subsection{General exponential terms}

Corollary~\ref{C:iteration-four} covers every parenthesization with
at most four occurrences of $x$. The recursive family in
Theorem~\ref{T:iteration-reconstruction} and the substitution rule
also give reconstructing terms with arbitrarily many occurrences.
These results do not establish reconstruction for every exponential
term.

The paper \cite{GRAA} announces an extension to every
exponential term with at most twelve occurrences of $x$.
The general question remains.

\smallskip
\noindent\textbf{Question.}
Is every exponential term reconstructing?

\subsection{Computational complexity}

Our reconstruction determines an isomorphism type. It does not
provide a~running-time bound for computing that type from an
explicitly given order relation.

\smallskip
\noindent\textbf{Question.}
Is there an algorithm which, given the order relation of an
$N$-element ordered set $P$ known to be isomorphic to $A^A$ for some
finite ordered set $A$, constructs an ordered set isomorphic to $A$
in time polynomial in $N$?

\smallskip
Here the input is the ordered set $P$. Its labels do not
identify its elements as maps on an already known $A$. We may take
the representation to be an $N\times N$ order-relation matrix;
polynomial time in $N$ is equivalent to polynomial time in the size
of this representation. The promise that such an $A$ exists is part
of the question. Recognizing whether an arbitrary $P$ is a~self-map
ordered set is a~different problem.

Some terminating reconstruction procedure exists. Constant maps
show that $|A|\leq|A^A|=N$. We can enumerate all ordered sets of size
at most $N$, form their self-map ordered sets, and test isomorphism
with $P$. The promise ensures that a~candidate is found, and
Theorem~\ref{T:main} makes its isomorphism type unique. This argument
does not give a~polynomial bound.

The size of the main intermediate ordered sets is controlled.
For a~positive-height input $E\cong A^A$, the ordered sets in
Proposition~\ref{P:reconstruction} satisfy
\begin{equation}\label{E:reconstruction-sizes}
\begin{aligned}
   |U|&\leq\height(E)\leq|E|-1,\\
   |V|&\leq|U|,\\
   |W|&\leq\height(V)<|V|.
\end{aligned}
\end{equation}
Indeed, \eqref{E:Lambda-size} bounds $\Tow(\Max(E))$ and
$\Tow(V)$, while selecting an interval does not increase
cardinality. Thus the displayed intermediate structures are smaller
than the input. This does not bound the work required to find them.

The operations whose cost we have not analyzed include direct
product decomposition, computation of the normal forms, comparison
in the auxiliary ordering, and extraction of the factors used in
cancellation. Implementing the particular selection rule would
require an effective enumeration of isomorphism types or another
effective comparison procedure. A~different reconstruction method
might avoid that rule. The question above concerns any algorithm,
not just an implementation of this proof.

\section*{Acknowledgments and disclosures}

\subsection*{Use of AI}
The authors used an AI system to assist in checking the proofs and
improving the exposition. This assistance included suggestions for
examples, organization, and wording. The authors are responsible
for the mathematical content and the final text.

\subsection*{Research funding}
Neither author received research support for this article.

\end{document}